\documentclass[11pt]{amsart}
\usepackage{amsmath,amssymb,amsthm}
\usepackage{graphics}
\usepackage{graphicx}
\usepackage{epstopdf}
\usepackage{hyperref}

\newtheorem{theorem}{Theorem}
\newtheorem{lemma}{Lemma}

\newtheorem{corollary}{Corollary}

\newcommand{\C}{\mathcal{C}}

\newcommand{\bb}{\mathbf{b}_1}
\newcommand{\bbr}{\mathbf{b}_1^r}
\newcommand{\bfb}{\mathbf{b}_2}
\newcommand{\bfbr}{\mathbf{b}_2^r}
\newcommand{\Au}{\mathrm{Aut}}

\begin{document}

\parindent = 0cm
\parskip = .3cm

\title[Symmetric Chain Orders]{Some Quotients of the Boolean Lattice are\\ Symmetric Chain Orders}

\author{Dwight Duffus}
\address {Mathematics \& Computer Science Department\\
   Emory University, Atlanta, GA  30322, USA}
   \email[Dwight Duffus]{dwight@mathcs.emory.edu}
   
\author{Jeremy McKibben-Sanders}
\address {Mathematics \& Computer Science Department\\
   Emory University, Atlanta, GA  30322, USA}
\email[Jeremy McKibben-Saunders]{jmckib2@gmail.com}

\author{Kyle Thayer}
\address {Mathematics \& Computer Science Department\\
   Emory University, Atlanta, GA  30322, USA}
\email[Kyle Thayer]{kthayer@emory.edu}



\begin{abstract}
Canfield and Mason have conjectured that for all subgroups $G$ of the automorphism group of the
Boolean lattice $B_n$ (which can be regarded as the symmetric group $S_n$)
the quotient order $B(n)/G$ is a symmetric chain order.
We provide a straightforward proof of a generalization of a result of K. K. Jordan: namely,
 $B(n)/G$ is an SCO whenever $G$ is generated by powers of disjoint cycles.  The symmetric chain 
 decompositions of Greene and Kleitman provide the basis for partitions of these quotients.
\end{abstract}

\maketitle

\section{Introduction}\label{s:intro}

There are several familiar notions of symmetry for the family of finite ranked partially ordered sets.  This 
family can be defined in more general ways (see \cite{E}), but for our
purposes, all of our finite partially ordered sets $P$ have a minimum element $0_P$ and for all $x \in P$,
all saturated chains $C \subseteq P$ with minimum element $0_P$ and maximum $x$ have the same length 
$r_P (x) :\!= |C| - 1$.  Such $P$ are called {\it ranked} posets, $r = r_P$ is the {\it rank function} and $r(P)$,
the maximum over all $r(x), x \in P$, is the rank of $P$.
Note that a ranked ordered set satisfies the Jordan-Dedekind chain condition:  for all $x \le y$ in $P$, all 
saturated chains in the interval $[x, y]$ have the same length.

In a ranked order $P$ the chain $x_1 < x_2 < \dots < x_k$ is a \emph{symmetric chain} if it is saturated and if $r(x_1) + r(x_k) = r(P)$.  
A \emph{symmetric chain decomposition} or \emph{SCD} of $P$ is a partition of $P$ into symmetric chains.  
If $P$ has an SCD, call $P$ a {\it symmetric chain order}, or an {\it SCO}.  Here, we are concerned with ordered sets based on the \emph{Boolean 
lattice}, denoted $B_n$, which is the power set of $[n] = \{1, 2, \dots, n\}$ ordered by containment.  Clearly $B_n$ is a ranked poset, 
with $\emptyset$ being the minimum element, and $r(A) = |A|$ for all $A \subseteq [n]$.  In fact, it is an SCO \cite{BTK}.

We are interested in ordered sets defined by actions of the automorphism group of $B_n$.  It is well-known that this group is faithfully
induced by the symmetric group $S_n$ of all permutations on the underlying set $[n]$, so we will refer to $S_n$ as the automorphism 
group of $B_n$.  Given any subgroup $G$ of $S_n$, the {\it quotient} $B_n / G$ has as its elements
the orbits in  $B_n$ under $G$
$$[A] = \{ B \ | \ B = \sigma (A), \text{for some} \ \sigma \in G \},$$
$A \in B_n$, ordered by
$$[A] \le [B] \ \iff \ X \subseteq Y \ \text{for some} \ X \in [A]  \ \text{and}  \ Y \in [B] .$$
In studying Venn diagrams, Griggs, Killian and Savage \cite{GKS} explicitly constructed an SCD of the quotient $B_n/G$ for $n$ prime
and given that $G$ is generated by a single $n$-cycle.  They asked if this {\it necklace} poset is an SCO for arbitrary $n$.   Canfield and 
Mason \cite{CM} made a much more general conjecture:  for all subgroups $G$ of $S_n$, $B_n/G$  is a symmetric chain order.   

Jordan \cite{KKJ} gave a positive answer to the question of Griggs, Killian and Savage, basing the SCD of the quotient on the
explicit construction of an SCD in $B_n$ by Greene and Kleitman \cite{GK}.   The construction in \cite{KKJ} requires an intermediate
equivalence relation and some careful analysis.   Here we provide a more direct proof of a generalization of 
Jordan's theorem by  ``pruning" the Greene-Kleitman SCD.  More generally, we show that $B_n/G$ is an SCO provided that $G$ is generated by 
powers of disjoint cycles (see Theorem \ref{T:cycles}).    We also provide a different proof  that $B_n/G$ is an SCO when $G$ is a 2-element subgroup 
generated  by a reflection, based on an SCD of $B_{\lfloor n/2 \rfloor}$.

The ordered sets $B_n/G$ do share several forms of symmetry or regularity with the Boolean lattice.   An SCO $P$ is
necessarily rank-symmetric, rank-unimodal, and strongly Sperner (see, for instance, \cite{KKJ} for definitions).    A result of
Stanley \cite{RPS} shows that $B_n/G$ has these three properties for all subgroups $G$ of $S_n$.  However, these three conditions are not
sufficient to yield symmetric chain decompositions. 

On the other hand, Griggs \cite{JRG} showed that a ranked ordered set with the LYM property, rank-symmetry and rank-unimodality 
is an SCO.  It is not known whether the quotients $B_n/G$ have the LYM property in general, though they do
if $n$ is prime and $G$ is generated by an $n$-cycle, which gives the Griggs, Killian and Savage result.    (An SCO need not 
satisfy the LYM property -- see \cite{JRG} for examples.)   Pouzet and Rosenberg \cite{PR} obtain  Stanley's results and ``local" families 
of symmetric chains  for more general structures than the quotients $B_n/G$, but their results do not show that $B_n/G$ is an SCO.

\section{The Main Results}\label{s:main}

There are two results for quotients of $B_n$ by  groups generated by powers of disjoint cycles and for a particular 2-element group.  The third
result concerns quotients of powers of a finite chain.

\begin{theorem}\label{T:cycles}
Let $G$ be a subgroup of $S_n$ generated by powers of disjoint cycles.  Then the partially ordered set $B_n/G$ is a symmetric chain order.
\end{theorem}

The proof of Theorem \ref{T:cycles} follows from this sequence of results.   The new  proof of Lemma \ref{L:Jordan}, which is a modest generalization of
Jordan's result,  is given in Section {\ref{S:Jordan}.

\begin{lemma}\label{L:Jordan}
Let $\sigma$ be an $n$-cycle in $S_n$ and let $H$ be a subgroup of the group generated by $\sigma$. Then $B_n/H$
is a symmetric chain order.
\end{lemma}

The following fact is well-known and can be proved by an argument much like the original proof in \cite{BTK} that the divisor lattice
of an integer is an SCO.  (In \cite{E}, this credited to Alekseev \cite{A}.)

\begin{lemma}\label{L:product}
Let $P$ and $Q$ be partially ordered sets.  If $P$ and $Q$ are symmetric chain orders then so is $P \times Q$.
\end{lemma}

In the following lemma, use this notation.  Suppose that $\sigma_j$ $(j = 1, 2, \ldots , t)$ are disjoint cycles in $S_n$ and that
$\rho_j = \sigma_j^{r_j}$, for integers $r_1, r_2, \ldots,  r_t$.   Let
$X_j$ be the subset of $[n]$ of elements moved by $\rho_j$  $(j = 1, 2, \ldots , t)$, and let $X_0$ be all elements of $[n]$ fixed by
all the $\rho_j$'s.  Let $B(X)$ denote the Boolean lattice of all subsets of a set $X$.

\begin{lemma}\label{L:disjoint}
Let $H_j$ be the subgroup of $S_n$ generated by $\rho_j$  ($j = 1, 2, \ldots, t$) and let $G$ be the subgroup generated by  $\{ \rho_1 , \rho_2, \ldots, \rho_t \}$.  Then
$$B(n)/G \ \cong \ B(X_0) \times B(X_1)/H_1  \times \cdots \times B(X_t)/H_t .$$
\end{lemma}

\begin{proof}
For any $A \subseteq [n]$, let $[A]$ denote its equivalence class in $B(n)/G$, and let $A_j = A \cap X_j$, $j= 0, 1, \ldots, t$.
Define a map $\Phi$ on $B(n)/G$ by $\Phi([A]) = (A_0, [A_1], \ldots [A_t]).$  From the definition of the ordering of the quotient,
$[A] \le [B]$ in $B(n)/G$ if and only if there is some $\tau \in G$ such that $A \subseteq \tau(B)$.   Then 
$\tau = \rho_1^{i_1}  \rho_2^{i_2} \cdots \rho_t^{i_t}$ for nonnegative integers $i_1, i_2, \dots, i_t$.  The claimed 
isomorphism follows from this fact:
$$A \subseteq \tau(B) \ \text{if and only if} \ A_0 \subseteq B_0 \ \text{and} \ A_j \subseteq \sigma^{i_j} (B_j) \ \text{for}
\ j = 1, 2, \ldots, t \ .$$ \end{proof}

The following is actually a corollary of Theorem \ref{T:cycles}.    Indeed, a proof based on an approach like that used in the proof of Theorem
\ref{T:cycles}  -- a greedy pruning of a Greene-Kleitman SCD -- can be shown to provide a basis for the proof offered in Section \ref{S:reflection}.
However, the proof in Section \ref{S:reflection} provides some insight into the Greene-Kleitman SCD and may be of use for other choices for the
group of permutations, such as the dihedral group.

\begin{theorem}\label{T:reflection}
Let $G$ be a 2-element subgroup with non-unit element a product of disjoint transpositions.  Then the partially ordered set 
$B_n/G$ is a symmetric chain order.
\end{theorem}

The last result concerns quotients defined by automorphism groups of products of chains.  Given any partially ordered set $P$
and subgroup $G$ of its automorphism group $\Au (P)$, the quotient $P/G$ has elements the orbits $[x]$ on $P$ defined  by $G$ 
with $[x] \le [y]$ in $P/G$ if there are $x' \in [x]$ and $y' \in [y]$ such that $x' \le y'$ in $P$.  It follows from a result of Chang, J\'onsson and 
Tarski \cite{CJT}, on the strict refinement property for product decompositions of partially ordered sets, that for chains $C$, positive
integers $m$  and $\alpha \in \Au(C^m)$ that there is some  $\phi \in S_m$ such that 
$$ \alpha (c_1, c_2, \ldots , c_m ) = (c_{\phi^{-1}(1)}, c_{\phi^{-1}(2)}, \ldots, c_{\phi^{-1}(m)} ), \ \text{for all} \  (c_1, c_2, \ldots, c_m ) \in C^m .$$
In particular, automorphism groups of powers of chains behave as those of the Boolean lattice and we can regard $\Au(C^m)$ as the symmetric
group $S_m$ acting on the coordinates of $C^m$.

\begin{theorem}\label{T:chain}
Let $C$ be a chain and let $K$ be a subgroup of $S_m$ generated by powers of disjoint cycles.   Then $C^m/K$ is an SCO.
\end{theorem}

The proof, presented in Section \ref{S:chain}, is a consequence of the proof of Lemma \ref{L:Jordan} and some observations on the 
Greene-Kleitman SCD.   V. Dhand \cite{D} has a new, very interesting result that is more general than the essential part of Theorem \ref{T:chain}: if $P$ is any SCO then so is $P^n/{\mathbb Z_n}$.  His arguments depend upon algebraic tools.   

We note that Theorem \ref{T:chain} can be stated more generally for chain products.  Let  $P = \prod_{i = 1}^n C_i^{m_i}$ where $C_j \ncong
C_k$ for $j \ne k$.   The result of Chang {\it et al.} \cite{CJT} shows that each automorphism of $P$ factors into an $n$-tuple from
$\prod_{i = 1}^n \Au (C_i^{m_i} )$ and that each $ \Au (C_i^{m_i} ) \cong S_{m_i}$.  (See \cite{Du} for a proof of this.)  Thus, if $K$ is a 
subgroup of $\Au(P)$ which also factors into a product of subgroups of $S_{m_i}$ of the form covered by Theorem \ref{T:chain} then,
by  Lemmas \ref{L:product} and \ref{L:disjoint}, $P/K$ is an SCO.    In particular, we have this consequence.

\begin{corollary}\label{c:chain-corollary}    Let $P$ be a product of chains and let $K$ be a subgroup of $\Au (P)$ that is generated
by powers of disjoint cycles.  Then $P/K$ is an SCO.  
\end{corollary}

We use Corollary \ref{c:chain-corollary} to deal with some cases where $K$ does not factor so nicely in \cite{DT}.

\section{The Proof of Lemma \ref{L:Jordan}}\label{S:Jordan}

We use the natural order $1 < 2< \ldots < n$ on $[n]$ and may assume that the $n$-cycle $\sigma$ is $ (1  \ 2 \ \cdots \ n)$.   This is valid because any
$n$-cycle $\rho$ is a conjugate of $(1  \ 2 \ \cdots \ n)$ and for any subgroup $K$ of $S_n$ and any $\pi \in S_n$, 
$B_n/K \cong B_n/ \pi^{-1} K \pi$ via $[A] \mapsto [\pi(A)]$. 

  We first describe the procedure for
obaining an SCD of $B_n/H$ based on the Greene-Kleitman SCD of $B_n$ then verify that the procedure yields the claimed SCD.

Let $C_1, C_2 , \ldots, C_t$, where $t = \binom{n}{\lfloor n/2 \rfloor}$, be the symmetric chains in the Greene-Kleitman decomposition, ordered by decreasing length.    For all $A \in B_n$, $[A]$ is the equivalence class containing $A$ in $B_n/H$ where $H$ is the subgroup
of $S_n$ generated by $\rho = \sigma^s$.

{\bf Claim:} There is a family $\C = \{ C_{i_1}', C_{i_2}', \ldots C_{i_m}' \}$, with $(i_1, i_2, \dots, i_m)$ a subsequence of $(1, 2, \dots, t)$,
that satisfies these conditions:

\begin{description}

 \item[(3.1)]  for all $1 \le j \le m$, $C_{i_j}' \subseteq C_{i_j}$ and is a symmetric chain in $B_n$;\\
 
  \item[(3.2)]  for all $1 \le r < s \le m$ and  for all $A \in C_{i_r}', B \in C_{i_s}'$, $A \notin [B]$; and,\\
 
  \item[(3.3)]  for all $[X]$ there is some $Y \in [X]$ such that $Y \in C_{i_j}'$ for some $j$.
 
\end{description}  

For $j = 1, 2, \ldots , m$, let $\widehat{C}_j = \{ [A]  \ | \ A \in C_{i_j}' \}$.   Then the chains $\widehat{C}_1, \widehat{C}_2,  \ldots , \widehat{C}_m$
cover $B_n/H$ (by {\bf(3.3)}), the sets are disjoint  (by {\bf(3.2)}), and form symmetric chains  (by {\bf(3.1)}).   Thus, it is enough to verify
the {\bf Claim} in order to prove Lemma \ref{L:Jordan}.

Several properties of the Greene-Kleitman SCD of $B_n$ are needed.   For the most part, these are well-known -- see, for instance,
the descriptions in \cite{E} and \cite{KKJ}.   It is useful to regard members of $B_n$ both as subsets of $[n]$ and as binary sequences of length $n$,
defined with respect to the natural order.  (Indeed, one needs to fix an order to speak of {\it the} Greene-Kleitman SCD.) The  SCD is obtained by a bracketing or 
pairing procedure that has several equivalent descriptions.  Here are two that are useful to us.  Let $A \subseteq [n]$.

 \begin{description}
   \item[(3.4)]   If $1 \notin A$ and $2 \in A$, pair 1 and 2; define $p_A(2) = 1$.  Suppose that we have considered
   $1, 2, \ldots , k-1$.  If $k \in A$ and there is some $j < k$, $j \notin A$ such that $j$ is unpaired, then let $p_A(k)$ be the
  maximum such $j$ and say $p_A(k)$ and $k$ are {\it paired}.   Continue for all $k$ in $[n]$.\\        
   \item[(3.5)]  For all $x \in A$ such that precisely half of the elements of the interval $[y, x]$ are members
  of $A$, for some $1 \le y < x$, let $p_A(x)$ be the maximum such $y$.
\end{description}  
   
Let $R(A)$ be the set of all $x$ for which $p_A(x)$ is defined,  let $L(A) = \{ p_A(x) \ | \ x \in R(A) \}$, and let $P(A) = L(A) \cup R(A)$.  
Now set
$$f(A) =  A \cup \{z\}  , \ z = \min ( [n] - (A \cup L(A))) ,$$
if $ [n] - (A \cup P(A)) \ne \emptyset$; otherwise $f(A)$ is undefined.  Then this rule inverts $f$:
$$f^{-1} (B) = B - \{z\} , \ z = \max ( B - R(B)) . $$
Let $\C(A) = \{ f^k (A) \ | \ k \in \mathbb{Z} \}$.   As $A$ runs over all  of $B_n$, the distinct $\mathcal{C}(A)$'s provide the
Greene-Kleitman SCD of $B_n$.   

Then the following hold for all $A \in B_n$.

\begin{description}
 \item [(3.6)]  For all $x \in R(A)$, $[p_A(x), x] \subseteq P(A)$.\\
 \item [(3.7)]  $\C(A) = \{ X \in B_n \ | \ R(X) = R(A) \}$ and $p_X(a) = p_A(a)$ for all $X \in \C(A)$ and for all $a \in R(A)$.\\
 \item [(3.8)]  $\min (\C(A)) = R(A)$, $\max (\C(A)) = [n] - L(A)$; in fact, $\C(A)$ is the chain
 $$R(A) \subset R(A) \cup \{ a_1 \} \subset R(A) \cup \{ a_1, a_2\} \subset \ldots \subset R(A) \cup \{ a_1, a_2, \ldots a_t\} =  [n] - L(A),$$
 where $[n] - (R(A) \cup L(A)) = \{a_1 < a_2 < \ldots < a_t \}$.
\end{description}

The following two lemmas provide properties of this SCD that substantiate the  {\bf Claim}.   Given a symmetric chain
$C$ in $B_n$ and $X \in C$ with $|X| \le \lfloor n/2 \rfloor$, let $X^*$ to be the member of $C$ with $|X^*| = n - |X|$. 

\begin{lemma}\label{L:symmetric} 
For $i = 1, 2, \ldots , t$ and for all $X \in C_i$ with $|X| \le \lfloor n/2 \rfloor$, $(\sigma(X))^* = \sigma(X^*)$.   Thus $(\sigma^j(X))^* = \sigma^j(X^*)$ for
all integers $j$, so $(\rho(X))^* = \rho(X^*)$ for all $\rho \in H$.
\end{lemma}
A special case of the preceding lemma is in \cite{DHR}.  Since this reference is a technical report and the result does not appear to be available in the literature,
we prove this below.  The following is, to our knowledge, new and a proof is provided as well.

\begin{lemma}\label{L:tau}  Let $w \in \{1, 2, \dots, t\}$, and let $A \in C_w$  with $|A| \le \lceil n/2 \rceil$.  Suppose that there is some $B \in [A]$ such that $B \in C_j$ for some $j < w$.  Then
there is some $k < w$ and $D \in C_k$ such that $D \in [f^{-1}(A)]$, provided that $f^{-1}(A)$ is defined.
\end{lemma}

To prove the {\bf Claim} from these facts, define $\mathcal{C}$ inductively.

First, let $i_1 = 1$ and $C_{i_1}' = C_1$.  Suppose that  $C_{i_1}', C_{i_2}', \dots, C_{i_k}'$ are defined.
If there exists $i \in \{i_{k}+1 , \ldots, t \}$ such that for some $X \in C_i$,
\begin{equation}\label{E:bigcup}
[X] \cap (\bigcup_{j = 1}^{k} C_{i_j} ) = \emptyset
\end{equation}
let $i_{k+1}$ be the least such $i$ and let
\begin{equation}\label{E:prime}
C_{i_{k+1}}' = \{ Y \in C_{i_{k+1}} \ | \ [Y] \cap ( \bigcup_{j = 1}^{k} C_{i_j}') = \emptyset \}.
\end{equation}
 If there is no such $i$ then $m = k$ and the procedure is complete.
 
If $Y \in  C_{i_{k+1}}'$, with $|Y| \le  \lfloor n/2 \rfloor$ then $Y^* \in  C_{i_{k+1}}'$, by Lemma \ref{L:symmetric}.  Also, if
$Z \in C_{i_{k+1}}$ and $Y \subseteq Z \subseteq Y^*$ where $Y \in C_{i_{k+1}}'$ then $Z \in C_{i_{k+1}}'$ by Lemma \ref{L:tau}
and Lemma \ref{L:symmetric}.  Thus, $C_{i_{k+1}}'$ is symmetric in $B_n$ and {\bf (3.1)} holds.
Equation (\ref{E:prime}) verifies {\bf(3.2)}; {\bf (3.3)} follows from (\ref{E:bigcup}) and (\ref{E:prime}).

{\it Proof of Lemma \ref{L:symmetric}}.    The proof is divided into cases depending upon which of $R(X) \subseteq X \subseteq X^*$
contain $n$.    It is not possible that $n \in X - R(X)$, because $|X| \le \lfloor n/2 \rfloor$ means that for some $y < n$ precisely half
the elements of $[y, n]$ are in $X$, and, hence, $n \in R(X)$ by {\bf (3.5)}.   Consequently, there are three
cases.  In each case, we show that 
$$R(\sigma(X^*)) = R(\sigma(X)),$$ 
apply {\bf (3.7)} to see that $\sigma(X^*)$ and
$(\sigma(X))^*$ are both members of $\C(\sigma(X))$, and conclude that $\sigma(X^*) = (\sigma(X))^*$ since these sets both have
cardinality $n - |\sigma(X)|$.

{\bf Case 1:}  $n \notin X^*$

Since $n \in  [n] - L(X) = \max (\C(X)) $ and $n \notin X^*$, $X \ne \min (\C(X)) = R(X)$.   Thus, there exists $y = \min (X - R(X))$.   
If $y = 1$ then $p_{\sigma(X)}(2) = 1 = p_{\sigma(X^*)}(2)$.  For each $z \in R(X)$, $\sigma(z) = z + 1 \in R(\sigma(X))$ and each
$z + 1 \in R(\sigma(X))$ has $z \in R(X)$ apart from $z + 1 = 2$.  Thus,
\begin{align*}
   R(\sigma(X)) &= \sigma(R(X) \cup \{2\} &\text{since} \ p_{\sigma(X)} (2) = 1,\\
                           &= \sigma(R(X^*)) \cup \{2\} &\text{by} \ {\bf (3.7)},\\
                           &=R(\sigma(X^*))  &\text{since} \ p_{\sigma(X^*)} (2) = 1.
\end{align*}
If $y > 1$ we claim that $[1, y-1] \subseteq P(X)$.  Note that $y - 1 \in X$ as otherwise $y \in R(X)$ with $p_X(y) = y - 1$, contradicting
the choice of $y$.    By the minimality of $y$, $y - 1 \in R(X)$ and, by {\bf (3.7)}, $[p_X(y-1), y-1] \subseteq P(X)$.   Continue in the
same manner, with $p_X(y-1)-1$ in place of $y$, and thereby verify the claim that $[1, y-1] \subseteq P(X)$.
The argument is just about the same as when $y = 1$ except we use the fact that $[1, y-1] \subseteq P(X)$ and $1 \notin X$, so, $p_{\sigma(X)}(y+1) = 1$:
$$R(\sigma(X^*)) =     \sigma(R(X^*)) \cup \{y+1\}  = \sigma(R(X)) \cup \{y+1\}  = R(\sigma(X)) .$$
                         
{\bf Case 2:} $n \in R(X)$

Every element of $R(\sigma(X))$ is in $\sigma(R(X))$ and every element of $\sigma(R(X))$, except for 1, is in $R(\sigma(X)$.  Thus,
$$R(\sigma(X)) =  \sigma(R(X)) - \{1\} = \sigma(R(X^*)) - \{1\} = R(\sigma(X^*)).$$

{\bf Case 3:} $n \in X^* - X$

Since $n \in X^* - X$, {\bf (3.8)} shows that $X^* = \max(\C(X)) = [n] - L(X)$ and, thus,  $X = \min(\C(X)) = R(X)$. 

If $z + 1 \in R(\sigma(X))$ then $z \in X = R(X)$, so $R(\sigma(X)) \subseteq \sigma(R(X))$.  Conversely, $1 \notin \sigma(R(X))$,
and any $z + 1 \in \sigma(R(X))$ is obviously a member of $R(\sigma(X))$.  Thus, $R(\sigma(X)) = \sigma(R(X))$.  Similarly, since
$n \notin R(X)$, it follows that $R(\sigma(X^*)) = \sigma(R(X^*))$.  Hence, $R(\sigma(X^*)) = R(\sigma(X))$. 

Since  $(\sigma(X))^* = \sigma(X^*)$ for all $X$ with $|X| \le \lfloor n/2 \rfloor$, we can apply induction on $j$ to conclude that $(\sigma^j(X))^* = \sigma^j(X^*)$:\\

\centerline{$(\sigma^j(X))^* = (\sigma(\sigma^{j-1}(X)))^* = \sigma((\sigma^{j-1}(X))^*) = \sigma(\sigma^{j-1}(X^*)) = \sigma^j(X^*). \quad \square$}

\vspace{.3cm}
{\it Proof of Lemma \ref{L:tau}}.
As before, let $C_1, C_2 , \ldots, C_t$, where $t = \binom{n}{\lfloor n/2 \rfloor}$, be the symmetric chains in the Greene-Kleitman 
decomposition, ordered by decreasing length, and let $\sigma = (1 \ 2 \ \cdots  \ n)$.  Let $A \in C_w$ with $|A| \le \left \lceil \frac{n}{2} \right \rceil$.  Suppose that there exists a $j < w$ such that $B \in C_j$ and $B \in [A]$.  Hence there is an integer $r$ such that $B = \sigma^r(A)$.

Assume that $f^{-1}(A)$ is defined.  We show that there is a $k < w$ such that $D \in C_k$ and $D \in [f^{-1}(A)]$. Since $f^{-1}(A)$ is defined, $A - R(A) \ne \emptyset$.  Let $y = \max(A - R(A))$.   We may assume that $-(y-1) \le r \le n-y$, $r \ne 0$.  We consider two cases:

\textbf{Case 1:} $y + r \in R(B)$

Then $r > 0$ since otherwise $y$ would also be paired in $A$, contrary to its choice.    Each $z \in B$ with $y+r < z$ must be in
$R(B)$ since $y < z - r$ so, by the choice of $y$, $z - r \in R(A)$.  Now consider the binary sequence $\sigma^{r}(f^{-1}(A))$ contained in some chain $C_k$.  Recall that $f^{-1}(A) = A - \{y\}$ and note that $\sigma^{r}(f^{-1}(A)) = B - \{y + r \}$. It follows from this that $\sigma^{r}(f^{-1}(A))$ must have one fewer pairs than $B$, since $y + r$ will be unpaired in $\sigma^{r}(f^{-1}(A))$ while $y + r$ is paired in
$B$,  and there are no other differences in the pairings.    By {\bf (3.8)}, $|C_k| > |C_j|$, so $k < j < w$, as desired.

\textbf{Case 2:} $y + r  \in B - R(B)$

If $y + r = \max(B - R(B))$ then we are done, since then we have $f^{-1}(B) =\sigma^{r}(f^{-1}(A))$.  So suppose instead $z = \max(B - R(B))$
where $y + r < z$.  If $r  >  0$ then $z - r \in (A - R(A))$, contrary to the choice of $y$.  Thus $r < 0$.    If $z - r \le n$ then $z - r$ would be an
unpaired element of $A$, since $y < z - r$ remains unpaired in $A$.  This would contradict the choice of $y$.  Thus $n < z - r$.

We now prove that for some $p$, $\sigma^p(B)$ is the maximum element of its chain, and its chain is not a singleton.  This will contradict the fact that 
$|A| \le \left \lceil \frac{n}{2} \right \rceil$, since $|\sigma^p(B)| = |A|$.

Since $\sigma^{-r}(B) = A$, $-r > 0$, we obtain $A$ from $B$ by applying $\sigma$ $-r$ times.    Since $n < z - r$ there is some $p$ such that
$\sigma^p(z) = n$.  Let $X =   \sigma^p(B)$.  Then, $\sigma^p(z) \in X - R(X)$.    Because $n \in  X - R(X)$, $X = [n] - L(X)$.  By {\bf (3.8)}, $X$ is the maximum element of its chain.  The chain containing $X$ is not a singleton, since $X - R(X) \ne \emptyset$. $\square$

\section{The proof of Theorem \ref{T:reflection}}\label{S:reflection}

Let $\rho = (i_1 j_i ) (i_2 j_2 ) \cdots  (i_k j_k)$, where the transpositions are pairwise disjoint, let $X = \bigcup_{r = 1}^k \{ i_r, j_r \}$, and let
$G = \{1, \rho \}$.  Then
$$B_n / G \cong B(X)/G \times B([n] - X) $$
via the mapping $[A] \mapsto ( [A \cap X], A - X )$ for all $A \subseteq [n]$.   By Lemma \ref{L:product}, we may assume that $n$ is even
and that $n = 2k$.  Using the remark about conjugation at the beginning of Section \ref{S:Jordan}, we may assume that 
$\rho = (1 \ 2k)(2 \ 2k-1) \cdots (k \ k+1)$.  As noted in the introduction, $G$ can be generated by a power of a $2k$-cycle, so Theorem \ref{T:cycles}
applies.  (In fact,  $\rho = \tau^{-1} \sigma \tau$  where  $\sigma = (1 \ 2 \cdots 2k)$ and $\tau = (k+1 \ 2k)( k+ 2 \ 2k -1 ) \cdots$.~)
And, as we shall see, its proof method can be adapted to give the proof we offer here.  However, the argument below might help with
the most interesting open case, namely, showing that $B_n/D_{2n}$ is an SCO for the dihedral group $D_{2n}$.

Regard each $A \in B_{2k}$ as a concatenated pair of binary strings of length $k$.  That is, $A = \bb\bfbr$, where $\bb, \bfb \in \{0, 1 \}^k$ and 
$\mathbf{b}^r$ is the reverse of the binary $k$-sequence $\mathbf{b}$.  Then the equivalence classes in $B_n/G$ are the sets
$\{ \bb\bfbr , \bfb\bbr \}$; these sets have 2 elements except in the case that $\bb = \bfb$.

Let $C_1, C_2 , \ldots, C_t$, where $t = \binom{k}{\lfloor k/2 \rfloor}$, be any symmetric chain decomposition of $B_k$, ordered by decreasing length.   We define a total ordering $\preccurlyeq$ on  $B_k = \{0, 1 \}^k$ as follows:
\begin{equation}\label{E:order}
  \mathbf{b}_r  \preccurlyeq \mathbf{b}_s \ \text{if} \ \mathbf{b}_r \in C_i, \mathbf{b}_s \in C_j, \ i < j, \ \text{or if} \  \mathbf{b}_r \subseteq \mathbf{b}_s \ \text {in} \ C_i
  \ \text{for some} \  i.
\end{equation}
For $1 \le i < j \le t$, let $P_{ij} = C_i \times C_j$, with the coordinate-wise ordering induced by the containment order on $B_k$, for
each $i = 1, 2, \ldots, t$, let
$$P_{ii} = \{( \mathbf{b}_r , \mathbf{b}_s )  \in C_i \times C_i \ | \ \mathbf{b}_r  \subseteq \mathbf{b}_s \} , $$
ordered coordinate-wise, and let 
$$P = \bigcup_{1 \le i \le j \le t} P_{ij} ,$$
again, ordered coordinate-wise.  Thus, $P$ is a subset of $B_{2k}$ with the exactly the ordering inherited from the Boolean lattice.  

In fact, with $r_P$ and $r_{B_k}$ as the rank functions in $P$ and $B_k$, respectively, then for $1 \le i \le j \le t$ and
with $r_{B_k}(\min C_i ) = r_i, \ \text{and} \ r_{B_k} ( \max C_i ) = k - r_i$,
$$r_P ( \min P_{ij} ) = r_i + r_j,  \  r_P ( \max P_{ij} ) = 2k - (r_i + r_j ), \ \text{and} \ l( P_{ij}) = 2k - 2(r_i + r_j) .$$
We see that each $P_{ij}$ is a symmetric subset of $B_{2k}$ in which the covering relation is preserved, that is,
$(  \mathbf{b}_p , \mathbf{b}_q )$ is covered by $(  \mathbf{b}_u , \mathbf{b}_v )$ in some $P_{ij}$ if and only if   $\mathbf{b}_p \mathbf{b}_q^r$ is covered by   
$\mathbf{b}_u \mathbf{b}_v^r$ in $B_{2k}$. 

Consider the map $\phi$ of $P$ to $B_{2k}/G$ defined by $\phi (  (\bb, \bfb ) ) =   \{ \bb\bfbr , \bfb\bbr\}$.   Since $ \bb \preccurlyeq \bfb$ for
all  $(\bb, \bfb ) \in P$, $\phi$ is injective.  It is obviously a surjection.  It is also order-preserving:  if $(  \mathbf{b}_p , \mathbf{b}_q ) \le  (  \mathbf{b}_u , \mathbf{b}_v )$  in $P$ then $\mathbf{b}_p\mathbf{b}_q^r \le  \mathbf{b}_u\mathbf{b}_v^r$ in $B_{2k}$.

Since the rank of an equivalence class in $B_{2k}/G$ is the rank of its members in $B_{2k}$, it follows that a symmetric chain in $P$ is a symmetric chain in $B_{2k}/G$.
Thus, it is enough to proof the following.

{\bf Claim:}  $P$ has a symmetric chain decomposition.

Since $P$ is partitioned by  $P_{ij}$, $1 \le i \le j \le t$, each of which preserve the covering relation in $P$, it is enough to prove that each $P_{ij}$ has a partition into
chains, each of which is symmetric in $P$.

For $1 \le i < j \le t$, $P_{ij} = C_i \times C_j$ is a cover-preserving subset of $P$, with minimum element at level $r_i + r_j$ and maximum element at level $2k - (r_i + r_l)$ in 
$P$, a partially ordered set of length $2k$.  Then the ``standard" symmetric partition of a product of two chains (the original partition in \cite{BTK})
provides symmetric chains in $P$.

For $i = 1, 2, \dots, t$, $P_{ii} = \{( \mathbf{b}_r , \mathbf{b}_s )  \in C_i \times C_i \ | \ \mathbf{b}_r  \subseteq \mathbf{b}_s \} , $ 
where $C_i$ is the chain of binary strings $\mathbf{b}_{r _i}\subset \mathbf{b}_{r_i +1} \subset \cdots  \subset \mathbf{b}_{k - r_i}$, where 
$r_{B_k} (\mathbf{b}_s ) = s$ in $B_k$, that is, is an $s$-element set, for $s = r_i, r_i+1, \ldots, k - r_i$.   Then $P_{ii}$ is an interval in $P$ with
minimum element at level $2r_i$ and maximum element at level $2k - 2r_i$ in $P$.  Also,
$$( \mathbf{b}_{r_i} , \mathbf{b}_{r_i} ) < ( \mathbf{b}_{r_i} , \mathbf{b}_{r_i+1} ) < \cdots < ( \mathbf{b}_{r_i} , \mathbf{b}_{k-r_i} ) < ( \mathbf{b}_{r_i+1} , \mathbf{b}_{k - r_i}) <\cdots <  ( \mathbf{b}_{k-r_i} , \mathbf{b}_{k-r_i})$$
is a symmetric chain in $P$ and $P_{ii} - C$ is a cover-preserving subset of $P$, isomorphic to the product of two chains, with minimum element
$( \mathbf{b}_{r_i+1} , \mathbf{b}_{r_i +1})$ at level $2r_i + 2$ and maximum element $ ( \mathbf{b}_{k-{r_i}-1} , \mathbf{b}_{k - {r_i}-1})$ at level
$2k - 2r_i - 2$.   By induction, we have a decomposition of $P_{ii}$ by chains symmetric in $P$.  This verifies  the {\bf Claim} and completes
the proof of Theorem \ref{T:reflection}.

\section{The proof of Theorem \ref{T:chain}}\label{S:chain}

With Lemmas  \ref{L:product} and \ref{L:disjoint}, it is enough to prove the result for $K$ generated by a single $m$-cycle.
We assume that $C$ is the $k$-element chain $00 \ldots 0, 10 \ldots 0, 11 \ldots 0, 111 \dots 1$ in the Boolean lattice $B_{k-1}$.
Let $n = (k-1)m$.  Then $C^m$ is the sublattice of $B_n$ consisting of all binary sequences of length $n$ of the form
$$\mathbf{b} = \mathbf{b_1}\mathbf{b_2} \ldots \mathbf{b_m}, \ \text{where each} \ \mathbf{b}_i \in C .$$
That is, the elements of $C^m$ are exactly those $n$-sequences which are $m$ $(k-1)$-sequences of 1's followed by 0's.

Let $C_1, C_2 , \ldots, C_t$, where $t = \binom{n}{\lfloor n/2 \rfloor}$, be the symmetric chains in the 
Greene-Kleitman SCD of $B_n$, ordered by decreasing length, as in Section \ref{S:Jordan}.   We claim that for each $j$,
$C_j \subseteq C^m$ or $C_j \cap C^m = \emptyset$.

Suppose that $\mathbf{b} = \mathbf{b_1}\mathbf{b_2} \ldots \mathbf{b_m} \in C_j \cap C^m$.   With the notation in {\bf (3.6) - (3.8)},
and applying these to $A = \mathbf{b}$, we can see that $i^{\mathrm{th}}$ entry $b_i$ of $\mathbf{b}$ is determined as follows:
\begin{equation*}
  b_i \ = \ 
  \begin{cases}
  0 		& i \in L({\mathbf b}) \cup \{a_{r}, a_{r+1}, \ldots , a_t \} \\
  1		& i \in R({\mathbf b}) \cup \{a_{1}, a_2, \ldots , a_{r-1}\} . \\
  \end{cases}
\end{equation*}
for some $r$.
  
If $\mathbf{b}$ is not the maximum element of $C_j$ then its successor $\mathbf{b'}$ is obtained by changing the 0 in position $a_r$
to a 1.   Either $a_r = 1$ or the entry in $\mathbf{b}$ in position $a_r - 1$ is a 1, by {\bf (3.6)}.  Thus, $\mathbf{b'}$ consists of
$m$ $(k-1)$-sequences of 1's followed by 0's and, so, belongs to $C_j \cap C^m$.  If $\mathbf{b}$ is not the minimum element of 
$C_j$ then its predecessor $\mathbf{b''}$ is obtained by changing the 1 in position $a_{r-1}$ to a 0.    Either $a_{r-1} = n$ or the
entry in $\mathbf{b}$ in position $a_{r-1} +1$ is a 0, by {\bf (3.6)}.  Again $\mathbf{b''}$ consists of
$m$ $(k-1)$-sequences of 1's followed by 0's and, so, belongs to $C_j \cap C^m$.  Hence, if $C_j \cap C^m \ne \emptyset$ then
$C_j \subseteq C^m$.

Let $K = \langle \phi^r\rangle$ where we may assume that $\phi = (1 2 \cdots m)$.  We need an SCD for $C^m/K$.   We know that $C^m$ is
a sublattice of $B_n$, as noted above, and that $\phi^r = {\sigma^{(k-1)r}}_{|C^m}$ where $\sigma = (1 2 \cdots n) \in S_n$.  
As in the proof of Lemma \ref{L:Jordan}, the {\bf Claim} gives an SCD $\widehat{C}_j = \{ [A]  \ | \ A \in C_{i_j}' \}$, $j = 1, 2, \dots , m$,
of $B_n/H$ where $H = \langle \sigma^{(k-1)r} \rangle$.   Thus, the subfamily 
$$ \widehat{C}_j = \{ [A]  \ | \ A \in C_{i_j}' \cap C^m\} ,  \ j = 1, 2, \dots , m,$$
is an SCD for  $C^m/K$.

\end{document}